\documentclass[11pt]{article}

\usepackage[letterpaper,left=1.25in,right=1.25in,top=1in,bottom=1in]{geometry}
\usepackage{amsmath,amssymb,amsthm}
\usepackage{setspace}
\usepackage{microtype}
\usepackage{needspace}
\usepackage{hyperref}

\allowdisplaybreaks
\newtheorem{theorem}{Theorem}[section]
\newtheorem{conjecture}[theorem]{Conjecture}
\newtheorem{lemma}[theorem]{Lemma}

\newtheorem*{restatedsmalltheorem}{Theorem~\ref{thm:small-explicit}}

\hypersetup{
  hidelinks,
  pdfauthor={Xinheng Lin},
  pdftitle={Coloring Small Kt-Minor-Free Graphs},
  pdfsubject={Graph coloring and complete minors},
  pdfkeywords={Hadwiger's conjecture, graph minors, chromatic number, independence number, star contractions}
}

\title{Coloring Small \texorpdfstring{$K_t$}{Kt}-Minor-Free Graphs}
\author{Xinheng Lin\thanks{%
  Email: \href{mailto:2500310007@fzu.edu.cn}{2500310007@fzu.edu.cn}}%
  \protect\\[1.4ex]
  \normalsize Center for Discrete Mathematics, Fuzhou University\protect\\[-0.1ex]
  \normalsize Fuzhou, P.\ R.\ China}
\date{September 8, 2026}

\begin{document}

\setlength{\abovedisplayskip}{6pt plus 2pt minus 2pt}
\setlength{\belowdisplayskip}{6pt plus 2pt minus 2pt}
\setlength{\abovedisplayshortskip}{0pt plus 2pt}
\setlength{\belowdisplayshortskip}{3pt plus 2pt minus 1pt}

\maketitle

\begin{abstract}
Delcourt and Postle proved that every $K_t$-minor-free graph is
$O(t\log\log t)$-colorable and reduced the Linear Hadwiger Conjecture to
coloring $K_t$-minor-free graphs on $O(t\log^4 t)$ vertices. In this paper,
we use the star contraction technique to improve their bound for small graphs
and use their reduction to extend
this improvement to all $K_t$-minor-free graphs. Thus we improve their
$O(t\log\log t)$ bound to $O(t\sqrt{\log\log t})$.

\end{abstract}

\section{Introduction}

All graphs in this paper are finite and simple. For a graph $G$, we denote
its chromatic number by $\chi(G)$. For graphs $H$
and $G$, we write $H\preccurlyeq_{\mathrm m} G$ if $H$ is a minor of $G$. We denote by $K_t$
the complete graph on $t$ vertices, and say that $G$ is $K_t$-minor-free if
$K_t\not\preccurlyeq_{\mathrm m} G$.

In 1943, Hadwiger~\cite{hadwiger} proposed the following conjecture.

\begin{conjecture}[Hadwiger's Conjecture~\cite{hadwiger}]\label{conj:hadwiger}
For every integer $t\geq1$, every $K_t$-minor-free graph is
$(t-1)$-colorable.
\end{conjecture}

Although 83 years have passed, Hadwiger's conjecture remains open for every
$t\geq7$. The case $t=5$ follows from Wagner's structure
theorem~\cite{wagner} and the Four Color Theorem~\cite{appel-haken,four-color}, while
the case $t=6$ was proved by Robertson, Seymour and
Thomas~\cite{robertson-seymour-thomas} in 1993. The difficulty of the
conjecture motivates the following linear weakening, raised by Reed and
Seymour~\cite{reed-seymour}.

\begin{conjecture}[Linear Hadwiger Conjecture~\cite{reed-seymour}]\label{conj:linear-hadwiger}
There exists an absolute constant $C_{\ref{conj:linear-hadwiger}}>0$ such that, for every integer
$t\geq1$, every $K_t$-minor-free graph $G$ satisfies $\chi(G)\leq C_{\ref{conj:linear-hadwiger}}t$.
\end{conjecture}

In the 1980s, Kostochka~\cite{kostochka} and Thomason~\cite{thomason}
independently proved that every $K_t$-minor-free graph has average degree
$O(t\sqrt{\log t})$, and hence is $O(t\sqrt{\log t})$-colorable. This
remained the best general bound for several decades. In 2019, Norin,
Postle and Song~\cite{norin-postle-song} broke this barrier by proving
that, for every fixed $\beta>1/4$, every $K_t$-minor-free graph is
$O(t(\log t)^\beta)$-colorable.
The strongest previously known general bound is due to Delcourt and
Postle~\cite{delcourt-postle}.

\begin{theorem}[Delcourt--Postle~\cite{delcourt-postle}]\label{thm:dp-coloring}
There exists an absolute constant $C_{\ref{thm:dp-coloring}}>0$ such that, for every integer
$t\geq3$, every $K_t$-minor-free graph $G$ satisfies
\[
  \chi(G)\leq C_{\ref{thm:dp-coloring}}t\log\log t.
\]
\end{theorem}

Delcourt and Postle~\cite{delcourt-postle} also reduced the Linear
Hadwiger Conjecture to coloring $K_t$-minor-free graphs on
$O(t\log^4 t)$ vertices. Note that the author~\cite{lin-small-connected}
improved this bound to $O(t\log^2 t)$. In Section~\ref{sec:preliminaries}, we establish an
improved bound for the chromatic number of such small graphs. Combining
this estimate with their reduction, we improve Theorem~\ref{thm:dp-coloring}
as follows.

\begin{theorem}\label{thm:global}
There exists an absolute constant $C_{\ref{thm:global}}>0$ such that, for every integer
$t\geq3$, every $K_t$-minor-free graph $G$ satisfies
\[
  \chi(G)\leq C_{\ref{thm:global}}t\sqrt{\log\log t}.
\]
\end{theorem}

We prove Theorem~\ref{thm:global} in Section~\ref{sec:main-proof}.

\section{Preliminaries}\label{sec:preliminaries}

All logarithms are natural. For $x>0$, we write
$\log^{+} x:=\max\{0,\log x\}$. For a graph $G$, let $\alpha(G)$ denote
its independence number, let $\delta(G)$ denote its minimum degree,
and write $v(G):=|V(G)|$.
The neighborhood and degree of a vertex $v\in V(G)$ are denoted by
$N_G(v)$ and $\deg_G(v)$, respectively. For $X\subseteq V(G)$, let $G[X]$
be the subgraph induced by $X$.

Following Diestel~\cite[Section~1.6]{diestel}, a \emph{star} is a graph
of the form $K_{1,r}$. If $c\in V(G)$ and
$L\subseteq N_G(c)$ is a nonempty independent set, then
$G[\{c\}\cup L]$ is an induced
star with \emph{center} $c$ and \emph{leaf set} $L$. Note that when considering several
vertex-disjoint induced stars, we allow edges between different stars.

If $A\subseteq V(G)$ is nonempty and $G[A]$ is connected, then $G/A$
denotes the simple graph obtained from $G$ by contracting $A$ to a single
vertex, deleting loops, and suppressing parallel edges. We denote the
resulting vertex by $v_A$.

We use the following consequence of the theorem of Duchet and
Meyniel~\cite{duchet-meyniel}, which plays a vital role in coloring small
$K_t$-minor-free graphs.

\begin{theorem}[Duchet--Meyniel~\cite{duchet-meyniel}]\label{thm:minor-independence}
Let $t\geq2$ be an integer, and let $G$ be a $K_t$-minor-free graph.
Every nonempty minor $H$ of $G$ satisfies
\[
  \alpha(H)\geq\frac{v(H)}{2(t-1)}.
\]
\end{theorem}

Theorem~\ref{thm:minor-independence} shows that every nonempty
$K_t$-minor-free graph contains a large independent set. Combining
Theorem~\ref{thm:minor-independence} with
\cite[Corollary~3.2]{delcourt-postle}, with $p:=2(t-1)$, gives the
following bound.

\begin{theorem}[Delcourt--Postle~\cite{delcourt-postle}]\label{thm:greedy-bound}
Let $t\geq2$ be an integer, and let $G$ be a nonempty $K_t$-minor-free
graph. Then
\[
  \chi(G)\leq2(t-1)\left(2+
  \log^{+}\!\left(\frac{v(G)}{2(t-1)}\right)\right).
\]
\end{theorem}

The idea of Delcourt and Postle's proof of Theorem~\ref{thm:greedy-bound}
is to repeatedly delete a large independent set. At each step, the
chromatic number decreases by at most one, while the number of vertices
decreases by a fixed proportion. If the initial chromatic number were
too large, this process would eventually leave fewer vertices than the
chromatic number, a contradiction. Our idea is similar, but we repeatedly
contract stars instead of deleting independent sets. This gives the
following result.

\begin{theorem}\label{thm:small-explicit}
Let $t\geq2$ be an integer, and let $G$ be a nonempty $K_t$-minor-free
graph with $v(G)\leq4(t-1)^2$. Then
\[
  \chi(G)\leq4(t-1)+8(t-1)
  \sqrt{\log^{+}\!\left(\frac{v(G)}{4(t-1)}\right)}.
\]
\end{theorem}

We prove Theorem~\ref{thm:small-explicit} after Lemma~\ref{lem:abstract-small}.

Lemma~\ref{lem:lift} below shows that contracting several vertex-disjoint
induced stars whose centers form an independent set decreases the
chromatic number by at most one.

\begin{lemma}\label{lem:lift}
If $G$ contains pairwise vertex-disjoint induced stars
$G[\{c_i\}\cup L_i]$, for every $1\leq i\leq m$, such that
$\{c_1,\ldots,c_m\}$ is an independent set, and $H$ is the simple graph
obtained by contracting each star to a single vertex, then
\[
  \chi(G)\leq\chi(H)+1.
\]
\end{lemma}

\begin{proof}
Let $x_i:=v_{\{c_i\}\cup L_i}$ for every $1\leq i\leq m$, and take
a proper $\chi(H)$-coloring of $H$. Give every vertex of $L_i$ the color
of $x_i$, and retain the colors of vertices outside the stars. Then, give
all centers $c_i$ a single new color. Clearly, this is a proper coloring of $G$.
\end{proof}

Let $G$ be a graph with large minimum degree such that every nonempty
minor of $G$ contains a large independent set. We aim to construct
vertex-disjoint induced stars in $G$ with many leaves in total.

\begin{lemma}\label{lem:packing}
Let $r\geq2$ and $i\geq2r+1$ be integers. Let $G$ be a nonempty graph with
$v(G)\leq r^2$ and $\delta(G)\geq i-1$. Suppose that every nonempty minor
$H$ of $G$ satisfies $\alpha(H)\geq v(H)/r$.
Then $G$ contains pairwise vertex-disjoint induced stars whose centers
form an independent set and whose leaf sets have total size at least
\[
  \frac{v(G)(i-1)}{8r^2}.
\]
\end{lemma}

\begin{proof}
Put
\[
  m:=\left\lfloor\frac{v(G)}{r}\right\rfloor,
  \qquad s:=\left\lfloor\frac{i-1}{2r}\right\rfloor.
\]
Since $v(G)\geq i\geq2r+1$, we have $m,s\geq1$.
Since $\alpha(G)\geq v(G)/r$, we can choose an independent set
$I=\{c_1,\ldots,c_m\}$ of size $m$.

Then we greedily construct $L_1,\ldots,L_m$ in this order so that each
$L_j$ has size $s$.
Since $|N_G(c_1)|\geq i-1$ and
\[
  \alpha(G[N_G(c_1)])\geq\frac{|N_G(c_1)|}{r}
  \geq\frac{i-1}{r}\geq s,
\]
we can choose an independent set $L_1\subseteq N_G(c_1)$ of size $s$.
Since $I$ is independent, $L_1\cap I=\varnothing$.
For $2\leq j\leq m$, suppose that $L_1,\ldots,L_{j-1}$ have already been
chosen and are pairwise
disjoint and disjoint from $I$. Put
\[
  W_j:=N_G(c_j)\setminus\bigcup_{\ell<j}L_\ell.
\]
Since $I$ is independent, $W_j$ contains no center. Moreover, $v(G)\leq r^2$
implies
\[
  ms\leq\frac{v(G)}{r}\frac{i-1}{2r}
  \leq\frac{i-1}{2}.
\]
Consequently,
\[
  |W_j|\geq i-1-(j-1)s\geq i-1-ms\geq\frac{i-1}{2}>0.
\]
Since
\[
  \alpha(G[W_j])\geq\frac{|W_j|}{r}
  \geq\frac{i-1}{2r}\geq s,
\]
we can choose an independent set $L_j\subseteq W_j$ of size
$s$. This completes the construction.

Hence
\[
  \sum_{j=1}^m|L_j|=ms
  \geq\frac{v(G)}{2r}\frac{i-1}{4r}
  =\frac{v(G)(i-1)}{8r^2},
\]
as required.
\end{proof}

A \emph{proper minor} of a graph $H$ is a minor not isomorphic to $H$.
Lemma~\ref{lem:critical} below provides the minimum-degree condition
required in Lemma~\ref{lem:packing} without decreasing the chromatic
number. It uses the classical critical-graph argument originating in the
work of Dirac~\cite{dirac-critical}.

\begin{lemma}[Dirac~\cite{dirac-critical}]\label{lem:critical}
Let $G$ be a nonempty graph. Then $G$ has a nonempty minor $H$ such that
$\chi(H)=\chi(G)$ and $\delta(H)\geq\chi(G)-1$, and every proper minor
of $H$ has chromatic number at most $\chi(G)-1$.
\end{lemma}

\begin{proof}
Choose a minor $H$ of $G$ that is minor-minimal subject to $\chi(H)=\chi(G)$.
Such a choice exists because $G$ is a minor of $G$.
If a proper minor $J$ of $H$ has $\chi(J)\geq\chi(G)$, then $J$ contains an
induced subgraph with chromatic number $\chi(G)$, contradicting the minimality of $H$.
Thus every proper minor of $H$ has chromatic number at most $\chi(G)-1$.
Moreover, $H$ is nonempty. If $\chi(G)=1$, the conclusion is immediate.
Now suppose $\chi(G)\geq2$. If some vertex $v\in V(H)$ satisfies
$\deg_H(v)\leq\chi(G)-2$, a proper $(\chi(G)-1)$-coloring of $H-v$ extends to $v$,
a contradiction.
\end{proof}

Starting from a graph $G$ with $v(G)\leq r^2$, we repeatedly apply
Lemma~\ref{lem:critical} and
contract vertex-disjoint induced stars whose centers form an independent
set. Lemmas~\ref{lem:lift}--\ref{lem:critical} ensure that the chromatic
number decreases by exactly one at each step, while the number of vertices
decreases by a proportion depending on the current chromatic number.
If $\chi(G)$ is too large, this process yields a graph with fewer vertices
than its chromatic number, a contradiction. This gives the upper bound
on $\chi(G)$ in the following lemma.

\begin{lemma}\label{lem:abstract-small}
Let $r\geq2$ be an integer, and let $G$ be a nonempty graph with
$v(G)\leq r^2$. Suppose that every nonempty minor $H$ of $G$ satisfies
$\alpha(H)\geq v(H)/r$. Then
\[
  \chi(G)\leq2r+4r\sqrt{\log^{+}\!\left(\frac{v(G)}{2r}\right)}.
\]
\end{lemma}

\begin{proof}[Proof of Lemma~\ref{lem:abstract-small}]
If $\chi(G)\leq2r$, the desired bound is
immediate. We may therefore assume that $\chi(G)>2r$.

Apply Lemma~\ref{lem:critical} to $G$, obtaining a minor
$F_{\chi(G)}$ such that $\chi(F_{\chi(G)})=\chi(G)$ and
$\delta(F_{\chi(G)})\geq\chi(G)-1$, and every proper minor of
$F_{\chi(G)}$ has chromatic number at most $\chi(G)-1$.
Then inductively construct a sequence of minors $F_{\chi(G)},F_{\chi(G)-1},\ldots,F_{2r}$
of $G$ such that, for every $2r\leq i\leq\chi(G)$, we have $\chi(F_i)=i$
and $\delta(F_i)\geq i-1$, and every proper minor of $F_i$ has chromatic
number at most $i-1$.
Moreover, for every $2r+1\leq i\leq\chi(G)$,
\[
  v(F_{i-1})\leq v(F_i)\left(1-\frac{i-1}{8r^2}\right).
\]

Fix an integer $i$ with $2r+1\leq i\leq\chi(G)$, and suppose that $F_i$
has been constructed. Then, by the induction hypothesis, $F_i$ is a minor of $G$,
$\chi(F_i)=i$, $\delta(F_i)\geq i-1$,
and every proper minor of $F_i$ has chromatic number at most $i-1$.
By Lemma~\ref{lem:packing}, $F_i$ contains
pairwise vertex-disjoint induced stars with independent centers and with
a total of $L_i$ leaves, where
\[
  L_i\geq\frac{v(F_i)(i-1)}{8r^2}.
\]
Contract these stars simultaneously to obtain a minor $H_i$.
Lemma~\ref{lem:lift} gives $\chi(H_i)\geq i-1$. Since at least one leaf
is contracted, $H_i$ is a proper minor of $F_i$. Hence $\chi(H_i)\leq i-1$,
and consequently $\chi(H_i)=i-1$.

Apply Lemma~\ref{lem:critical} to $H_i$ to obtain
$F_{i-1}$ such that $\chi(F_{i-1})=\chi(H_i)=i-1$ and
$\delta(F_{i-1})\geq(i-1)-1$,
and every proper minor of $F_{i-1}$ has chromatic number at most $(i-1)-1$.
Its order satisfies
\[
  v(F_{i-1})\leq v(H_i)=v(F_i)-L_i
  \leq v(F_i)\left(1-\frac{i-1}{8r^2}\right).
\]
Moreover, $F_{i-1}\preccurlyeq_{\mathrm m} H_i\preccurlyeq_{\mathrm m} F_i\preccurlyeq_{\mathrm m} G$, so the induction
hypothesis holds at the next step. This completes the
construction down to $F_{2r}$.

Since $\chi(F_{2r})=2r$, we have $v(F_{2r})\geq2r$.
For every $2r\leq i\leq\chi(G)$, $i\leq v(F_i)\leq r^2$, and hence
$0<(i-1)/(8r^2)<1$. By
$1-x\leq\exp(-x)$, we obtain
\begin{align*}
  2r&\leq v(F_{2r})
  \leq v(G)\prod_{i=2r+1}^{\chi(G)}\left(1-\frac{i-1}{8r^2}\right)
  \leq v(G)\exp\left(-\frac{1}{8r^2}
      \sum_{i=2r+1}^{\chi(G)}(i-1)\right)\\
  &=v(G)\exp\left(-\frac{\chi(G)(\chi(G)-1)-2r(2r-1)}{16r^2}\right).
\end{align*}
Taking logarithms and rearranging gives
\[
  \chi(G)(\chi(G)-1)\leq2r(2r-1)+16r^2\log\frac{v(G)}{2r}.
\]

Put $z:=\log(v(G)/(2r))$. Since $v(G)\geq \chi(G)>2r$, we have $z>0$. Observe that
\begin{align*}
  (2r+4r\sqrt z)(2r+4r\sqrt z-1)
  &=2r(2r-1)+(16r^2-4r)\sqrt z+16r^2z\\
  &\geq2r(2r-1)+16r^2z.
\end{align*}
The function $x(x-1)$ is strictly increasing for $x\geq1/2$.
Hence $\chi(G)\leq2r+4r\sqrt z$, as required.
\end{proof}

For convenience, we restate Theorem~\ref{thm:small-explicit} below.

\begin{restatedsmalltheorem}
Let $t\geq2$ be an integer, and let $G$ be a nonempty $K_t$-minor-free
graph with $v(G)\leq4(t-1)^2$. Then
\[
  \chi(G)\leq4(t-1)+8(t-1)
  \sqrt{\log^{+}\!\left(\frac{v(G)}{4(t-1)}\right)}.
\]
\end{restatedsmalltheorem}

\begin{proof}[Proof of Theorem~\ref{thm:small-explicit}]
Apply Lemma~\ref{lem:abstract-small} with $r:=2(t-1)$.
Theorem~\ref{thm:minor-independence} gives $\alpha(H)\geq v(H)/(2(t-1))$
for every nonempty minor $H$ of $G$.
\end{proof}

\section{Proof of main result}\label{sec:main-proof}

We use the following reduction of Delcourt and
Postle~\cite[Theorem~1.6]{delcourt-postle}.

\begin{theorem}[Delcourt--Postle~\cite{delcourt-postle}]\label{thm:dp}
There exists an integer $C_{\ref{thm:dp}}\geq1$ such that the following holds:
Let $t\geq3$ be an integer. Let $G$ be a nonempty graph and let
\[
  f_{\ref{thm:dp}}(G,t):=\max_{H\subseteq G}\left\{
    \frac{\chi(H)}{a}:\ t\geq a\geq\frac{t}{\sqrt{\log t}},\ 
    v(H)\leq C_{\ref{thm:dp}}a\log^4 a,\ H\text{ is $K_a$-minor-free}
  \right\}.
\]
If $G$ is $K_t$-minor-free, then
\[
  \chi(G)\leq C_{\ref{thm:dp}}t\bigl(1+f_{\ref{thm:dp}}(G,t)\bigr).
\]
\end{theorem}

The idea behind the proof of Theorem~\ref{thm:dp} is as follows. If $G$
has sufficiently large chromatic number, then they can find sufficiently
many highly connected subgraphs, each containing a relatively small
complete minor model. They then use linkage techniques to join these
models into a large complete minor model.

\begin{proof}[Proof of Theorem~\ref{thm:global}]
Fix the absolute integer $C_{\ref{thm:dp}}$ in Theorem~\ref{thm:dp}.
We can choose a sufficiently large integer $C\geq1$ such that, for every integer $a\geq3$,
\[
  C_{\ref{thm:dp}}a(\log a)^4\leq4(Ca-1)^2.
\]

Choose a pair $(a,H)$ attaining the maximum in the definition of $f_{\ref{thm:dp}}(G,t)$.
Since $t/\sqrt{\log t}>2$ for $t\geq3$, we have $3\leq a\leq t$.
By the choice of $H$,
\[
  v(H)\leq C_{\ref{thm:dp}}a(\log a)^4\leq4(Ca-1)^2.
\]
The graph $H$ is $K_a$-minor-free and hence also $K_{Ca}$-minor-free.
Applying Theorem~\ref{thm:small-explicit} with parameter $Ca$ gives
\[
  \chi(H)\leq4(Ca-1)+8(Ca-1)
  \sqrt{\log^{+}\!\left(\frac{v(H)}{4(Ca-1)}\right)}.
\]
Since $4(Ca-1)\geq a$, $C_{\ref{thm:dp}}\geq1$, and $3\leq a\leq t$, it follows that
\begin{align*}
  f_{\ref{thm:dp}}(G,t)=\frac{\chi(H)}a
  &\leq4C+8C\sqrt{\log^{+}\!\bigl(C_{\ref{thm:dp}}(\log a)^4\bigr)}
  =4C+8C\sqrt{\log C_{\ref{thm:dp}}+4\log\log a}\\
  &\leq4C+8C\sqrt{\log C_{\ref{thm:dp}}+4\log\log t}.
\end{align*}
Theorem~\ref{thm:dp} now yields
\[
  \chi(G)\leq C_{\ref{thm:dp}}t\left(1+4C+8C\sqrt{\log C_{\ref{thm:dp}}+4\log\log t}\right).
\]
Since $\log\log t\geq\log\log3>0$ for $t\geq3$, and $C_{\ref{thm:dp}}$ and $C$ are
absolute constants, there exists an absolute constant $C_{\ref{thm:global}}>0$ such that
\[
  \chi(G)\leq C_{\ref{thm:global}}t\sqrt{\log\log t}
\]
for every integer $t\geq3$, as required.
\end{proof}

\medskip
\noindent\textbf{Statement of AI Use.}
The central proof idea in this manuscript was developed with the assistance of
GPT 5.6 Sol and GPT 6 Astra. The author subsequently refined the idea and wrote
the manuscript, using the same models to assist with language polishing.


\begin{thebibliography}{99}

\bibitem{appel-haken}
K. Appel and W. Haken,
\emph{Every planar map is four colorable},
Contemporary Mathematics \textbf{98},
American Mathematical Society, Providence, RI, 1989,
\href{https://bookstore.ams.org/conm-98/}{AMS book page}.

\bibitem{delcourt-postle}
M. Delcourt and L. Postle,
\emph{Reducing linear Hadwiger's conjecture to coloring small graphs},
J. Amer. Math. Soc. \textbf{38} (2025), no.~2, 481--507,
\href{https://doi.org/10.1090/jams/1047}{doi:10.1090/jams/1047}.

\bibitem{diestel}
R. Diestel,
\emph{Graph Theory}, sixth edition,
Springer, 2025,
\href{https://diestel-graph-theory.com/}{author's book website}.

\bibitem{dirac-critical}
G. A. Dirac,
\emph{A property of 4-chromatic graphs and some remarks on critical graphs},
J. London Math. Soc. \textbf{27} (1952), no.~1, 85--92,
\href{https://doi.org/10.1112/jlms/s1-27.1.85}{doi:10.1112/jlms/s1-27.1.85}.

\bibitem{duchet-meyniel}
P. Duchet and H. Meyniel,
\emph{On Hadwiger's number and the stability number},
in B. Bollob\'as (ed.), Graph Theory,
North-Holland Mathematics Studies \textbf{62},
North-Holland, 1982, 71--73,
\href{https://doi.org/10.1016/S0304-0208(08)73549-7}
{doi:10.1016/S0304-0208(08)73549-7}.

\bibitem{hadwiger}
H. Hadwiger,
\emph{\"Uber eine Klassifikation der Streckenkomplexe},
Vierteljschr. Naturforsch. Ges. Z\"urich \textbf{88} (1943), 133--142.

\bibitem{kostochka}
A. V. Kostochka,
\emph{Lower bound of the Hadwiger number of graphs by their average degree},
Combinatorica \textbf{4} (1984), no.~4, 307--316,
\href{https://doi.org/10.1007/BF02579141}{doi:10.1007/BF02579141}.

\bibitem{lin-small-connected}
X. Lin,
\emph{An asymptotically tight $t\log t$ bound for $k$-connected subgraphs in dense $K_t$-minor-free graphs},
preprint, 2026,
\href{https://arxiv.org/abs/2607.21222v3}{arXiv:2607.21222v3}.

\bibitem{norin-postle-song}
S. Norin, L. Postle and Z.-X. Song,
\emph{Breaking the degeneracy barrier for coloring graphs with no $K_t$ minor},
Adv. Math. \textbf{422} (2023), article 109020,
\href{https://doi.org/10.1016/j.aim.2023.109020}{doi:10.1016/j.aim.2023.109020}.

\bibitem{reed-seymour}
B. Reed and P. Seymour,
\emph{Fractional colouring and Hadwiger's conjecture},
J. Combin. Theory Ser. B \textbf{74} (1998), no.~2, 147--152,
\href{https://doi.org/10.1006/jctb.1998.1835}{doi:10.1006/jctb.1998.1835}.

\bibitem{four-color}
N. Robertson, D. Sanders, P. Seymour and R. Thomas,
\emph{The four-colour theorem},
J. Combin. Theory Ser. B \textbf{70} (1997), no.~1, 2--44,
\href{https://doi.org/10.1006/jctb.1997.1750}{doi:10.1006/jctb.1997.1750}.

\bibitem{robertson-seymour-thomas}
N. Robertson, P. Seymour and R. Thomas,
\emph{Hadwiger's conjecture for $K_6$-free graphs},
Combinatorica \textbf{13} (1993), no.~3, 279--361,
\href{https://doi.org/10.1007/BF01202354}{doi:10.1007/BF01202354}.

\bibitem{thomason}
A. Thomason,
\emph{An extremal function for contractions of graphs},
Math. Proc. Cambridge Philos. Soc. \textbf{95} (1984), no.~2, 261--265,
\href{https://doi.org/10.1017/S0305004100061521}{doi:10.1017/S0305004100061521}.

\bibitem{wagner}
K. Wagner,
\emph{\"Uber eine Eigenschaft der ebenen Komplexe},
Math. Ann. \textbf{114} (1937), 570--590,
\href{https://doi.org/10.1007/BF01594196}{doi:10.1007/BF01594196}.

\end{thebibliography}
\end{document}